\documentclass{article}
\usepackage{biblatex}
\usepackage{rotating}

\usepackage{amsmath, amssymb,amscd,geometry,amsthm}
\usepackage{thmtools}
\usepackage{color}
\usepackage[usenames,dvipsnames,svgnames,table]{xcolor}
\usepackage{epsfig}
\usepackage{soul}
\setstcolor{red}
\sethlcolor{yellow}

\usepackage{enumitem}
\usepackage{cprotect}
\usepackage{tikz}
\usepackage{quiver}
\usepackage{hyperref}

\theoremstyle{definition}

\newtheorem{thm}{Theorem}[section]
\newtheorem{propo}{Proposition}[section]
\newtheorem{defn}{Definition}[section]
\newtheorem{corollary}{Corollary}[section]
\newtheorem{example}{Example}[section]

\renewcommand{\iff}{\leftrightarrow}

\renewcommand{\models}{\vDash}

\newcommand\LL{\mathcal{L}}

\newcommand\RR{\mathbb{R}}
\newcommand\K{\mathbb{K}}

\newcommand\Q{\mathbb{Q}}

\newcommand\PP{\mathbb{P}}

\newcommand\MM{\mathcal{M}}
\newcommand\NN{\mathcal{N}}

\input{fitch.sty}
\numberwithin{equation}{subsection}
\usepackage{lipsum}
\bibliography{references}

\usepackage[utf8]{inputenc}

\title{The Model Theory of Falsification}
\author{Reid Dale \\ \href{mailto:reiddale@berkeley.edu}{reiddale@berkeley.edu}}
\date{September 2022}

\begin{document}

\maketitle

\begin{abstract}
    This paper is concerned with the question of when a theory is \textit{refutable} with certainty on the basis of sequence of primitive observations. Beginning with the simple definition of falsifiability as the ability to be refuted by \textit{some} finite collection of observations, I assess the literature on falsification and its descendants within the context of the dividing lines of contemporary model theory. The static case is broadly concerned with the question of how much of a theory can be subjected to falsifying experiments. In much of the literature, this question is tied up with whether the theory in question is axiomatizable by a collection of universal first-order sentences. I argue that this is too narrow a conception of falsification by demonstrating that a natural class of theories of distinct model-theoretic interest---so-called NIP theories---are themselves highly falsifiable. 
\end{abstract}

\section{Introduction}

Popper's \autocite{popper2005logic} solution to the demarcation problem says that the distinguishing feature of a scientific theory---construed as an empirical hypothesis---is its \textit{falsifiability}. Various accounts of falsification have emerged over the years; in this chapter, I aim to provide a model-theoretic account of falsification that will aid us in understanding both long-run and short-run properties of various falsificationist strategies.

Broadly speaking, falsification centers on the following question: given a class $\mathbb{K}$ of possible worlds, is there some finite collection of observations about our world $W$ that would allow us to infer $W\notin \mathbb{K}$? If the answer is ``yes,'' the class $\mathbb{K}$ is said to be falsifiable. 

Throughout, we suppose that $\mathcal{L}$ is a signature\footnote{Not necessarily finite or relational.} and $\mathbb{K} \subseteq \operatorname{Str}(\mathcal{L})$ a class of $\mathcal{L}$-structures. Epistemically, $\mathcal{L}$ plays the role of the collection of observable relations, functions, and constants that relate objects in the world. We suppose that the world $W$ is itself an $\mathcal{L}$-structure.

An \textit{observable formula} $\varphi(x_1,\dots,x_n)$ corresponds to a finite Boolean combination of atomic $\LL$-formulas. We say that $\varphi(x_1,\dots,x_n)$ is $\K$-forbidden just in case no $W\in \K$ realizes $\varphi$. Model theoretically, this is the same as saying
\[\K \models \neg (\exists x_1,\dots x_n)\varphi(x_1,\dots,x_n).\]
Of course, this is equivalent to 
\[\K \models (\forall x_1,\dots x_n)\neg\varphi(x_1,\dots,x_n),\]
motivating the following definition:
\begin{defn}
Let $\K$ be a class of $\mathcal{L}$-structures. We denote the universal theory of $\K$ as \[\forall_1(\K) = \{ \varphi\,|\, \varphi \text{ is a } \forall_1 \text{ first-order $\LL$-sentence and }\K \models \varphi \}.\]
We say that $\K$ is \textit{falsifiable} provided that the class of models of $\forall_1(\K)$ is nontrivial, i.e.,
\[\operatorname{Mod}(\forall_1(\K)) \neq Str(\LL). \]

Let $\K \subseteq \K'$ be two classes of $\LL$ structures. We say that $\K$ is falsifiable relative to $\K'$ provided the inclusion
\[\forall_1(\K) \supset \forall_1(\K')  \]
is proper. 
\end{defn}

One may think of $\operatorname{Str}(\mathcal{L})$ as the class of possible worlds relative to a signature $\LL$; indeed, it is the largest such collection of possible worlds. However, if one wants to study falsifiability relative to a class of $\LL$-structures containing analytic truths beyond the logical validities, one may turn to the relative definition of falsifiability. In the vast majority of this chapter, we will concern ourselves with falsification simpliciter.

An immediate corollary of this definition is that if $\K$ is falsifiable, then any class \textit{stronger} than $\K$ is falsifiable:

\begin{propo}\label{prop:downward_fals}
Let $\K$ be a falsifiable class. If $\K'\subseteq \K$, then $\K'$ is falsifiable.
\end{propo}
\begin{proof}
If $\K'\subseteq\K$ then $\forall_1(\K)\subseteq \forall_1(\K')$. Since $\forall_1(\K)$ contains nontrivial universal sentences, so does $\forall_1(\K')$, so $\K'$ is falsifiable.
\end{proof}

In particular, per this definition a class of structures $\K$ need not even be first-order axiomatizable in order to be falsifiable. Per this definition, all that is required of a class of structures to be falsifiable is that there is \textit{some} observable formula $\varphi$ whose realization is incompatible with the class $\K$. 

While this base notion of falsifiability is rather simple in its expression, refinements of the notion of falsification contain a great deal of mathematical complexity. In the remainder of this chapter, we will investigate refinements of falsifiability along two key axes:
\begin{enumerate}
    \item The \textbf{Static Case}, wherein we ask ``just how falsifiable is the class $\K$?'' relative to the above definition of falsifiability, and
    \item The \textbf{Dynamic Case}, wherein we ask ``how much observation is required to falsify a hypothesis, and how quickly can we expect to falsify it?''
\end{enumerate}

The present paper focuses on the \textit{static} case of falsification; a subsequent paper will address the \textit{dynamic} case.

In discussing the static case, we will primarily discuss variants surrounding the \textit{first-order, universal axiomatization} of a class. After all, if a class $\K$ is specified by a universal set of axioms, this means that each one of the theory's axioms expresses the class $\K$ being incompatible with some observation. Others, such as Simon and Groen \autocite[]{simon1979fit} and Chambers et al.~\autocite{chambers2014axiomatic}, propose even more stringent constraints on $\K$ than its universal axiomatizability to deem them falsifiable.

\section{Falsifiability and Unfalsifiability in Mechanics and Economics}

As a warm-up to our investigation of falsification, we investigate falsificational phenomena in physics and economics by way of an analysis of Newtonian Mechanics and the theory of choice.

\subsubsection{Newtonian Mechanics}

We begin by showing that, in a strong sense, the \textit{framework} of Newtonian Mechanics is unfalsifiable relative to the class of kinematic motions. 

For the definition of Newtonian system I follow the formalism given by Arnold~\autocite{arnol2013mathematical}.

\begin{defn} \autocite[p.~8]{arnol2013mathematical}
An $n$-particle \textit{motion} is a smooth function $x:\RR \to \RR^{3n}$ such that the graphs of the trajectories of each particle are non-intersecting.

A Newtonian system of $n$ particles is a motion $x:\RR \to \RR^{3n}$ such that there exists a vector field \[F: \RR^{3n}\times \RR^{3n} \times \RR \to \RR^{3n}\] such that
\[x''(t) = F(x(t),x'(t),t) \]
for all $t\in \RR$. 
\end{defn}

Let $\K$ be the class of Newtonian systems. Note here that since we do not require $\K$ to be an elementary class in order to be falsifiable, we do not have to exhibit a first-order axiomatization of $\K$.

By an $n$-particle \textit{kinematic datum} $e$ I mean an equality
\[ (x,x',x'')(t_0) = v\]
or inequality
\[ (x,x',x'')(t_0) \neq v \]
where $v \in \RR^{9n}$ and $t_0 \in \RR$. 
Intuitively, a kinematic datum is a specification of the numerical values of the vector $(x(t),x'(t),x''(t),t)\in \RR^{9n}\times \RR$.
For a set $E$ of kinematic data, let $\pi_t(E)$ be the set of times occurring as values in $E$. For an element $e\in E$, let $t_e$ be the value of the time coordinate of $e$.

We say that a set $E$ of kinematic data is \textit{motional} provided all sentences in $E$ are satisfied by some motion. %Clearly, a necessary and sufficient condition for a finite set $E$ of kinematic data to be motional is that for every $e,e' \in E$, if $t_e = t_{e'}$ then
%\begin{enumerate}
%    \item if $e$ and $e'$ are both equalities, 
%    \[e \leftrightarrow e', \]
%    \item if $e$ is an equality and $e'$ is an inequality, then %    \[e' \not\leftrightarrow \neg e,\]
%    and
%    \item no condition of $e$ and $e'$ are both inequalities.
%\end{enumerate}

\begin{propo}
Let $E$ be a finite motional set of kinematic data. Then there is an $n$-particle Newtonian system $x$ such that for all $e\in E$, $x$ satisfies all the conditions set out by $E$.

Thus, the class of Newtonian systems of $n$ particles is unfalsifiable relative to the class of $n$-particle motions.
\end{propo}

\begin{proof}
First, we note that if $E$ is motional we may replace all inequalities of $E$ with equalities to prove the claim.\footnote{Some care must be taken to ensure that the set $E$ remains motional in this case. So long as we replace $x_i(t) \neq p$ with some $x_i(t) = p'$ where $p'$ does not appear in the remaining conditions in $E$, the set $E'$ will remain motional. Since $E$ is finite and $\RR$ is infinite, it is always possible to replace finitely many inequalities with finitely many equalities and remain motional.} Let $Y_i$ denote the trajectory of the $i^{\text{th}}$ particle.

Thus, $e$ is equivalent to a system of equations of the form
\[(Y_i, Y'_i, Y''_i)(t_j) = v_j = (p_{ij}(t_j),v_{ij}(t_j),a_{ij}(t_j))\]
where $p_{ij}$, $v_{ij}$, and $a_{ij}$ represent the position, velocity, and trajectory of the $i^{\text{th}}$ particle at time $t_j$.

Since the data $E$ is motional, there exist $n$ smooth functions $Y_i:\RR \to \RR^3$ such that the positions of the $i^{\text{th}}$ particle $Y_i$ satisfy
\[Y_{i}(t_j) = p_{ij}.\]

We now show that we may alter this trajectory to ensure that $Y'_i(t_j) = v_{ij}$ and $Y''_i(t_j) = a_{ij}$ for each $i,j$. The following argument ensures that we can locally alter the $Y_i$'s without intersecting the graphs of the $Y_i$.

Let $I$ be a closed interval of finite length containing the open interval \newline $[\min((t_j)), \max((t_j))]$. Since the $Y_i$ are all continuous, there exists a compact box\footnote{By a box I mean a product of intervals.} $R \subseteq \RR^4$ such that  a neighborhood of each graph $\Gamma_i$ of each $Y_i$ restricted to $I$ is contained in $R$.

Since $R$ is a compact metrizable space and the graphs $\Gamma_i$ are closed and disjoint, there exists an $r\in \RR$ such that the \textit{tubular neighborhoods} of the graphs $\Gamma_i$ given by
\[ U_{i}(r) = \{x\in \RR^4 \,|\, d(x, \Gamma_i) < r \} \]
are disjoint. 

Now, by the existence and uniqueness of ODEs there locally exists a unique solution to the initial value problem
\[(x_i, x'_i, x''_i)(t_j) = (p_{ij},v_{ij},a_{ij}).\]
For each particle $i$, by taking a small enough interval $I_{i,j}$ around $t_j$ we let $Z_{i,j} :I_{i,j} \to \RR^3$ be the solution to this ODE and have the graph of $Z_{i,j}$ contained in the neighborhood $U_{i}(r)$ constructed above. Perhaps by shrinking the interval on which $Z_{i,j}$ solves the ODE, we may smoothly extend $Z_{i,j}$ to all of $\RR$ in a manner such that $Z_{i,j}(t) = 0$ for all $t \notin I_{i,j}$.

Now, by the existence of smooth bump functions, there exists a smooth bump function
\[b_{i,j}(t): \RR \to [0,1]\]
such that $b_{i,j}(t) = 0$ on some open interval containing $t_j$ and $b_{i,j}(t) = 1$ for all $t> max(I_{i,j})$ and $t<min(I_{i,j})$. We define
\[\tilde{Y}_i(t) = \left(\sum\limits_j b_{i,j}(t)\right)Y_i(t) + \left(1-\sum\limits_{j}b_{i,j}(t)\right) \sum\limits_{j}Z_{i,j}(t)) .\]
Then $\tilde{Y}_i(t)$ satisfies the required differential equations.

Finally, we must argue that there exists a force function $F:\RR^{3n}\times\RR^{3n}\times \RR \to \RR^{3n}$ such that
\[Y_i''(t) = F(Y_i(t),Y'_i(t),t)\]
for all $i$. Intuitively, we would like to define
\[F(x,x',t) = \sum\limits_{i} Y_i''(t) \text{ if } x(t) = Y_i(t) \]
but this is not a continuous function. However, by the construction of $U_r$, shrinking $r$ as necessary, we can ensure that this map is continuous by defining
\[F(x,x',t) = \sum\limits_{i} Y_i''(t) \text{ if } (x(t),t) \in U_i(r) .\qedhere\]

%We begin by proving the result in the case that $E$ is kinematic data for a single particle $P$ at two times $t_0, t_1 \in \mathbb{R}$. Suppose that we are given kinematic data $e_0 = (p,v,a)(t_0)$ and $e_1 = (p,v,a)(t_1)$.

%Let $Y_0$ and $Y_1$ be the solutions to the constant force equation $F_i \equiv a_i$ for $i\in \{0,1\}$. $Y_0$ and $Y_1$ are smooth functions defined on $\mathbb{R}$.Let $I_j = (s^{-}_j, s^{+}_j)$ be an interval around $t_j$ such that $I_0 \cap I_1 = \varnothing$. Shrinking $I_j$ as necessary, there is a smooth function $\tilde{Y}: (s^{-}_0, s^{+}_1) \to \mathbb{R}^3$ such that $\tilde{Y}|_{I_j} = Y_j$. Letting $F$ be the vector field given by $\partial^2 \tilde{Y}$ we have a model of Newtonian mechanics realizing the given kinematic data.

%For a single particle system with finitely-many time observations, we may repeat the above construction by shrinking $I_j$ finitely many times as necessary to define the trajectory $Y$.

%The only potential obstruction in the case of finitely many particles is that we must ensure that the motion $Y$ has the trajectories of all particles not intersect. This is a topological fact: we can smoothly deform the trajectory of any finite set of particles arbitrarily outside of the disjoint neighborhoods $I_j$ as in the above constructions, we can ensure that around each of the graphs $\Gamma_k$ of the trajectories $\tilde{Y}_k$ there is an open set $U_k$ in $\mathbb{R}^{4}$ separating each trajectory.

%\reid{Wes/Tom--How much detail to go in regarding the geometric construction here?}
\end{proof}

This theorem indicates that no matter how many finite points of data we collect regarding the \textit{kinematics} of the system, there will always be some Newtonian theory which accommodates that data. This is the sense in which the framework of Newtonian mechanics fails to be falsifiable.

However, there are natural strengthenings of the class $\K$ of Newtonian motions which are falsifiable.

For example, suppose that our hypothesis is that a particle $P$ is a free particle with motion $x$ with zero initial acceleration relative to a fixed observer's frame of reference. This implies that for all $t$ the resultant force $F(x,x',t)$ is identically zero. Thus, the motion $x$ must follow a straight line.

This hypothesis is highly falsifiable. Since a line in $\RR^3$ is determined by two points, the theory of the free particle entails that if $x(t_1)$ and $x(t_2)$ are the positions of the first two observations of the particle, all subsequent observations of the particle must be a member of the line $L(x(t_1),x(t_2))\in \RR^3$. Thus, for every $n>2$, each subsequent observation carries with it the chance of refuting the claim that the particle is free. This is an instance of the notions of \textit{always falsifiability} and \textit{VC finiteness}, which we will discuss in our treatment of the dynamic case of falsification in section 2.3.

\subsubsection{Theory of Choice}

We now turn our attention from the falsification of physical theories to the falsification of economic theories of choice. To keep things simple, we model \textit{preference} as a binary relation $\prec$ on a set of choices $C$, where $x\prec y$ is interpreted as ``$y$ is strictly preferred to $x$.'' The data $(C,\prec)$ is called a \textit{preference structure}.

A frequently assumed necessary condition for a preference to be considered \textit{rational} is that the preference relation is acyclic; namely, that there is no chain $x_1\prec x_2 \cdots x_n \prec x_1$ and $x_1\not\prec x_1$. Let $\K$ be the class of acyclic preference structures.

The class $\K$ is falsifiable: if one observes a configuration
\[ \left( \bigwedge\limits_{1\leq i < n} c_i \prec c_{i+1} \right) \wedge c_n \prec c_1  \]
then one can conclude that the underlying choice structure $(C,\prec) \notin \K $. In fact, $\K$ is universally axiomatizable, axiomatized by the collection 
\[ A_n = (\forall x_1,\dots, x_n) \neg \left(\left( \bigwedge\limits_{1\leq i < n} x_i \prec x_{i+1} \right) \wedge x_n \prec x_1\right) \]
for all $n\in \omega$.

On the other hand, there are common rationality assumptions which are \textit{not} falsifiable.

%For instance, the \textit{continuity axiom} \reid{Wes---any citations for this?}
%\[ \forall x \forall y (x\prec y \rightarrow \left( (\exists z) x\prec z \wedge z\prec y \right)),\]
%which says that if $y$ is preferred to $x$ then there exists an intermediate $z$ strictly between them in the preference order $\prec$. This axiom is \textit{not} falsifiable: given any choice structure $(C,\prec)$ and any pair $x\prec y$ there exists an extension \[(C\cup \{z\},\prec') \supseteq (C,\prec)\] such that $x\prec z \prec y$.

For example, consider the axioms invoked to prove the representability of a preference structure $(C,\prec)$ by a utility function $u:X\to \RR$, that is,
\[x \preceq y \iff u(x) \leq u(y). \]
Gilboa \autocite[p.~51]{gilboa2009theory} gives the axioms
\begin{enumerate}
    \item \textbf{Completeness}: $(\forall x, y)(x\preceq y \vee y\preceq x)$,
    \item \textbf{Transitivity}: $(\forall x,y,z) \left([x\prec y \wedge y\prec z] \rightarrow x\prec z \right)$, and
    \item \textbf{Separability}:$(\exists Z\subseteq X)(\forall x,y)(\exists z\in Z)((|Z|\leq \aleph_0) \wedge (x\prec y \rightarrow (x\preceq z \preceq y)) $.
\end{enumerate}

Completeness and Transitivity are both $\forall_1$ sentences in the language $\LL = \{ \prec\}$, but Separability is naturally expressed as a second-order sentence. 

Gilboa makes a very interesting argument for the admissibility of the Separability axiom: its unfalsifiability ``suggests that [Separability] has no empirical content and therefore does not restrict our theory... Rather, the axiom is a price we have to pay if we want to use a certain mathematical model'' \autocite[p.~52]{gilboa2009theory}. 

The following theorem makes this argument precise.

\begin{thm}
Let $\K = \operatorname{Mod}(T)$ be the class of models of an $\LL = \{\prec\}$-theory $T$. Let $\K_{sep}$ be the class $\MM \in \K$ satsifying Separability. Then $\K_{sep}$ is unfalsifiable relative to $\K$.
\end{thm}
\begin{proof}
Let $T$ be a first-order theory and $\varphi$ a $\forall_1$ sentence such that $\varphi\notin \forall_1(\K)$. We wish to show that $\K_{sep}\not\models \varphi$.

Since $\varphi\notin \forall_1(\K)$ there exists some $\MM\in \K$ such that $\MM \models \neg \varphi$. Since $\varphi$ is a universal sentence, $\neg \varphi$ is existential. Let $\overline{m} \in \MM^k$ witness $\neg \varphi$. By the L{\"o}wenheim-Skolem theorem there exists a countable elementary substructure $\MM'\preceq \MM$ containing $\overline{m}$ of size $\aleph_0$. Since $\MM'\in \K$ and is of size $\aleph_0$, $\MM'\in \K_{sep}$. Therefore $\K_{sep}\not\models \varphi$, so
$\K_{sep}$ is unfalsifiable relative to $\K$.
\end{proof}

Thus, not only is the Separability axiom unfalsifiable relative to the class of \textit{all} $\LL$-structures, it is unfalsifiable relative to any first-order axiomatizable theory of preference. In this way the Separability axiom is empirically harmless: we may freely adjoin the Separability axiom to any first-order theory $T$ of choice structures without inadvertently strengthening the observable consequences of $T$.

\section{Falsifiability and the Randomness of the Universe}

As an application of the results of the previous section, we argue that for many ways of making precise the assertion that ``the world is, at a fundamental level, random,'' the assertion is unfalsifiable.

Defining what it means to be ``random,'' however, poses a great difficulty. To this end, I consider two different formalizations of randomness as it pertains to structures:
\begin{enumerate}
    \item The evolution of the universe is \textit{generic} in a suitable sense, and
    \item The evolution of the universe is generated by a stochastic process.
\end{enumerate}

For a suitable formulation of each of the above cases we will see unfalsifiability arise. To make sense of these two notions we define the notion of a time-indexed structure.

\begin{defn}
Let $\mathcal{L}$ be a relational language. Then the \textit{time-indexed language} $\mathcal{L}_{\tau}$ is given by
\[\mathcal{L}_{\tau} = \{R_i(x_1,\dots,x_m,t)\,|\, R_i(x_1,\dots,x_m) \in \LL \} \cup \{O(x),\tau(x), < \}.\]

A \textit{time-indexed structure} is an $\mathcal{L}_{\tau}$ structure satisfying the theory $T_{\tau}$ given by the axioms:
    \begin{enumerate}
        \item Objects and Times are different sorts, i.e.,
        \[(\forall x)(O(x)\vee \tau(x) \wedge \neg(O(x)\wedge \tau(x)),\]
        \item $(\forall \overline{x},t)(R(\overline{x},t)  \rightarrow \left(\bigwedge O(x_i) \wedge \tau(t) \right) )$, and
        \item the relation $<$ is a linear order on $\tau$. \qedhere
    \end{enumerate}
    
\end{defn}

A time-indexed $\mathcal{L}_{\tau}$-structure can be thought of simply as a time-indexed family of $\mathcal{L}$-structures. We use such examples all the time:

\begin{example}
Let $\mathcal{L} = \{H(x)\}$ be the language consisting of a single unary predicate. We can regard an $\omega$-sequence of coin flips as an $\mathcal{L}_{\tau}$ structure in a straightforward manner. The domain of the structure $\MM = \omega \cup \{c\}$, where $c$ is an object corresponding to the coin. We interpret $O(\MM) = \{c\}$, $\tau(\MM) = \omega$, and $\MM \models H(c,n)$ just in case the $n^{\text{th}}$ coin flip of $c$ returns heads.
\end{example}

\subsubsection{The Richness of the Universe}

We first consider the notion of the randomness of the universe as specified by the notion of a Fra{\"i}ss{\'e} limit. 

\begin{defn} \autocite[pp.~321-322]{hodges1993model} 
A countable class $\K$ of finite $\LL$-structures is a Fra{\"i}ss{\'e} class provided $\K$ satisfies the following properties:
\begin{enumerate}
    \item Hereditary Property (HP): If $\MM\in \K$ and $\NN\subseteq \MM$ is finitely generated then $\NN\in \K$,
    \item Joint Embedding Property (JEP): If $\MM,\NN\in \K$ then there exists $\mathcal{Q}\in\K$ such that $\MM$ and $\NN$ both embed in $\mathcal{Q}$, and
    \item Amalgamation Property (AP): If $\MM,\NN,\mathcal{Q}$ are $\LL$-structures and $f_{\NN}:\MM\to \NN$, $f_{\mathcal{Q}}:\MM\to\mathcal{Q}$ are embeddings there exists a structure $\mathcal{S}$ and embeddings $g_{\NN}: \NN \to \mathcal{S}$ and $g_{\mathcal{Q}}: \mathcal{Q} \to \mathcal{S}$ such that
    \begin{equation*}g_{\NN}\circ f_{\NN} = g_{\mathcal{Q}}\circ f_{\mathcal{Q}}.\qedhere
    \end{equation*}

\end{enumerate}

\end{defn}

\begin{example}
The class of finite linear orders is a Fra{\"i}ss{\'e} class.
\end{example}

\begin{example}
Let $\LL$ be a finite relational language. The class of finite $\LL$-structures forms a Fra{\"i}ss{\'e} class.
\end{example}

For a Fra{\"i}ss{\'e} class $\K$, there is a unique, highly homogeneous, countable structure $\K_{lim}$ into which all and only the members of $\K$ embeds, called the Fra{\"i}ss{\'e} limit.

\begin{defn}
Let $\MM$ be an $\LL$-structure. The \textit{age} of $\MM$, $\operatorname{age}(\MM)$, is the class of all finitely-generated $\LL$-structures embeddable in $\MM$. 
\end{defn}

\begin{thm} \autocite[Theorem 7.1.2]{hodges1993model}
Let $\K$ be a Fra{\"i}ss{\'e} class of $\LL$ structures. Then there is an $\LL$ structure $\K_{lim}$, unique up to isomorphism, such that
\begin{enumerate}
    \item $\operatorname{age}(\K_{lim}) = \K$,
    \item $|\K_{lim}|\leq \aleph_0$, and
    \item every isomorphism between finitely generated substructures $\MM_1,\MM_2\subseteq \K_{lim}$ extends to an automorphism of $\K_{lim}$. \qedhere
\end{enumerate}
\end{thm}

Thus, a Fra{\"i}ss{\'e} limit is extremely rich, able to accommodate any finite number of observations. Moreover, when a Fra{\"i}ss{\'e} limit exists for a class $\K$, the first-order theory $\K_{lim}$ is $\forall_1$-conservative over $\K$.

\begin{propo}
Let $\K$ be a Fra{\"i}ss{\'e} class of $\LL$-structures where $\LL$ is a finite relational language. If $\varphi$ is a $\forall_1$ $\LL$-sentence, then 
\[\K_{lim} \models \varphi \iff \K \models \varphi.\qedhere\] 
\end{propo}
\begin{proof}
If $\K_{lim}\models \varphi$, then since every $\MM\in\K$ embeds into $\K_{lim}$ and $\varphi$ is $\forall_1$, $\K\models \varphi$.

Conversely, suppose that $\K_{lim} \not\models \varphi$. Then $\neg\varphi$ is existential, so there is some witness $\overline{m}\subset \K_{lim}$ to the falsity of $\varphi$. Since $\K = \operatorname{age}(\K_{lim})$ and $\left\langle m \right \rangle$ is a finitely-generated substructure of $\K_{lim}$, $\NN = \left\langle m \right \rangle \in \K$ and $\NN\models \neg \varphi$. Thus $\K\not\models \varphi$.\end{proof}

Thus, no new universal sentences are entailed by the Fra{\"i}ss{\'e} limit of $\K$. 

We show that the class of finite time-indexed structures forms a Fra{\"i}ss{\'e} class, which we shall see yields unfalsifiabilty of the generic theory relative to the class of time-indexed structures.

\begin{thm}
Let $T_{\tau}$ be the theory of time-indexed structures. Then
\begin{enumerate}
    \item the class of finite models of $T_{\tau}$ is a Fra{\"i}ss{\'e} class, and
    \item the theory $T_{\tau,lim}$ of the Fra{\"i}ss{\'e} limit of the class is the model companion of the theory $T_{\tau}$.\qedhere
\end{enumerate}
\end{thm}
\begin{proof}
We need to show that class of finite models of $T$ is a Fra{\"i}ss{\'e} class.

First, since the class $\operatorname{Mod}(T_{\tau})$ is universally axiomatizable, its finite models satisfy (HP).

By the axioms of $T_{\tau}$ each model $\mathcal{M} \in \operatorname{Mod}_{fin}(T_{\tau})$ can be expressed as
\[ \MM = ((W_i,t_i))_{i<m} \] 
where each $W_i$ is an $\LL$-structure (recall $\LL_{\tau}$ was obtained from $\LL$) and each $t_i$ is a time.

A necessary and sufficient condition for a map $f:\MM \to \NN$ with $\MM,\NN \in \operatorname{Mod}(T)$ is that $f$ restricted to $\tau$ is an order embedding and that for each time $t_i$, $f(W_i) \subset W_{f(t_i)}$ is an embedding of $\LL$-structures. From this decomposition of embeddings it is clear that the joint extension property and amalgamation property holds as the class of finite $\LL$-structures and the class of finite linear orders are both Fra{\"i}ss{\'e} classes: to jointly embed two finite models of $T_{\tau}$ structures, first jointly embed their temporal component and then jointly embed their $\LL$-structures at each time in the intersection of the embedding. Likewise, one may amalgamate by first amalgamating the temporal component and then amalgamating the $\LL$-structures over each time.

Thus the theory of the Fra{\"i}ss{\'e} limit $T_{\tau,lim}$ exists and model complete by \autocite[Theorem 7.4.2]{hodges1993model}. It remains to show that $T_{\tau,lim}$ is a model companion of $T_{\tau}$. 

Clearly every model of $T_{\tau,lim}$ is a model of $T_{\tau}$, so it suffices to show that every model of $T$ embeds into a model of $T_{\tau, lim}$. Suppose $\MM \models T_{\tau}$. Then since $T_{\tau}$ is $\forall_1$ axiomatizable, all finitely generated substructures of $\MM$ are models of $T_{\tau}$. Moreover, $\MM$ embeds into an ultraproduct of its finite substructures since the language is relational, and in turn each finite substructure embeds into the Fra{\"i}ss{\'e} limit of the class. Thus $\MM$ embeds into an ultrapower of the Fra{\"i}ss{\'e} limit of the class and hence, since $T_{lim}$ is elementary, a model of $T_{lim}$.
\end{proof}

As a corollary, we have the following.

\begin{corollary}
The theory $T_{\tau,lim}$ is $\forall_1$-conservative over $T$; thus, $T_{\tau,lim}$ is relatively unfalsifiable over $T_{\tau}$.
\end{corollary}

Therefore, for this sense of genericity, ``the Universe is a generic time-indexed $\LL$-structure'' is not falsifiable relative to the theory of time-indexed structures.

\subsubsection{The Stochasticity of the Universe}

We now turn to \textit{probabilistically generated} models of the evolution of the universe.

\begin{defn}

The \textit{discrete time-index language} $\mathcal{L}^d_{\tau}$ is
\[ \mathcal{L}^d_{\tau} = \{R_i(x_1,\dots,x_m,t)\,|\, R_i(x_1,\dots,x_m) \in \LL \} \cup \{O(x),\tau(x), <, S(x) \}.\qedhere\]

\end{defn}

We work with a distinguished class of $\mathcal{L}^d_{\tau}$ structures $\mathcal{M}$, namely those such that 
\begin{enumerate}
    \item $(\tau(\mathcal{M}),<,S)$ is the structure of $(\omega,<,S)$,
    \item  $O(\mathcal{M})$ is $[n]$ for some $n\in \omega$,
    \item the world $(W,0)$ is drawn from a probability distribution $\mu$ on the state space $\Sigma = Str_{\LL}([n])$, and
    \item the world $(W,t+1)$ is obtained from $(W,t)$ by way of a time-homogeneous memoryless Markov process, i.e., there exists a stochastic matrix $\rho$ on the state space $\Sigma$ in $\omega$ such that $\PP((W,t+1)|(W',t)) = \rho(W,W')$.
\end{enumerate}

\begin{example}
This construction generalizes the construction of an $\omega$ sequence of IID coin flips. In this case, $O(\MM) = \{c\}$ is a single object, and each time $t$ is associated to an $\LL = \{H(x)\}$-structure where $H(x)$ is a unary predicate meaning ``$x$ flipped heads.'' The basic $\LL_{\tau}^d$ predicate $H(x,t)$ means ``$x$ flipped heads at time $t$.''

Let $\mu$ be any measure on $\{H,T\}$, i.e., an assignment of $p_H,p_T \in [0,1]$ such that $p_H + p_T = 1$. The stochastic transition matrix $\rho$ is given explicitly by
\[ \rho =  \begin{pmatrix}
p_H & p_T  \\
p_T & p_H\\
\end{pmatrix}.
\]
Such a stochastic process generates an $\LL_{\tau}^d$ structure on domain $\omega \cup \{c\}$.\end{example}

% Also, we can't take $\mathbb{Z}$ because the only invertible stochastic matrices are the permutation matrices. NOT QUITE TRUE: AS LONG AS SOME POINT HAS A WELL_DEFINED DISTRIBUTION WE CAN MAKE SENSE OF IT

%It is crucial that we don't have a constant symbol naming zero, epistemically. We haven't pinned down $t_0$, but we know the successory relationships between times. This is very important for the structure of the $\forall_1$ sentences.

Now, let $\mathcal{C}$ be a set of pairs $(\mu, \rho)$ where $\mu$ is a probability distribution on $\Sigma$ and $\rho$ is a stochastic matrix on $\Sigma$.\footnote{Recall that a \textit{stochastic matrix} $\rho$ is a matrix such that the sum of the entries over each row and each column is $1$. A stochastic matrix $\rho$ is irreducible provided that for all $\sigma,\sigma'\in \Sigma$ there exists an $n\in \omega$ such that the $\rho^n(\sigma,\sigma') > 0$. In other words, each state is reachable from every other state after some finite number of steps with positive probability.} The choice of $\mu$ and $\rho$ induce a unique probability measure $\PP_{\mu,\rho}$ on $\Sigma^{\omega}$.
The existential $\mathcal{C}$-theory $T_\mathcal{C}$ in $\mathcal{L}^d_{\tau}$ is given by: 
\[ (\exists \overline{x},\overline{t}) \varphi(\overline{x},\overline{t}) \in T_{\mathcal{C}} \iff  (\forall \mu,\rho \in\mathcal{C}) \left( \PP_{\mu,\rho}(\varphi(\overline{x},\overline{t}) \text{ is eventually realized})= 1 \right).\]

Let $\mathcal{C}_{+}$ be the class of pairs $(\mu, \rho)$ of initial distributions $\mu$ on $\Sigma$ with $\mu(W) > 0$ for each $W\in \Sigma$ and stochastic matrices $\rho$ with rows and columns indexed by $\Sigma$ such that $\rho(W,W') > 0$ for all $W,W'\in \Sigma$. 

Let $\varphi_{\mathcal{M}}(\overline{x},t)$ be the sentence saying that at time $t$ the $\LL$-structure $\NN(t)$ is isomorphic to $\MM$. To show that every $\mathcal{L}_{\tau}^d$-satisfiable $\exists_1$ formula in the language $\mathcal{L}_{\tau}^d$ is a member of $T_{\mathcal{C}}$ it suffices to show that every formula of the form
\[ \left( \bigwedge\limits_{i} \varphi_{\MM_i}(\overline{x}_i,s_i) \right) \wedge \bigwedge\limits_{i} (s_i < s_{i+1}) \]
where 
\begin{enumerate}
    \item $W_i \in \Sigma$, 
    \item $s_i$ is a term in the language of the successor function $\{S(x)\}$, and
    \item the formula $\bigwedge\limits_{i} s_i < s_{i+1}$ is realizable in $\left\langle \omega, <, S(x)\right\rangle$
\end{enumerate}
is realized with probability one for each $\PP \in \mathcal{C}_{+}$.

We demonstrate this by studying an auxiliary Markov process on $\Sigma^m$, where $m$ is the number of terms $s_i$ occurring in the formula $\varphi$.

Let $(n_1,\dots,n_m)$ realize the formula $\bigwedge\limits_{i} s_i < s_{i+1}$. Note that for all $k\in\omega$, $(n_1+k,\dots,n_m+k)$ also realizes $\bigwedge\limits_{i} s_i < s_{i+1}$.

The stochastic process 
\[((W_1,n_1),\dots,(W_m,n_m)) \to ((W'_1,n_1+1),\dots,(W'_m,n_m+1))\] is inferred from the data $(\mu,\rho)$. The pair $(\mu,\rho)$ induce atime-homogeneous Markov Chain on $\Sigma^m$ as follows: for each $(W_1,\dots, W_m)$ assign initial probability \newline $\mu^{*}((W_1,\dots, W_m))$ according to the probability that $((W_1,n_1),\dots,(W_m,n_m))$ is realized given $(\mu,\rho)$ in the first $n_m$ transitions. Note that by assumptions on $\mu$ and $\rho$, 
\[\mu^*((W_1,\dots,W_m))>0 \]
for all $W_1,\dots, W_m$. Likewise, define a stochastic matrix $\rho^*$ on $\Sigma^m$ by setting 
\[\rho^*((W_1,\dots,W_m),(W'_1,\dots,W'_m)) = \prod\limits_{i} \rho(W_i,W'_i).\]
By assumption on $\rho$, $\rho^*(\sigma,\sigma') > 0$ for all $\sigma,\sigma' \in \Sigma^m$. Thus the data $(\mu^*,\rho^*)$ are themselves a time-homogeneous Markov process such that $\mu^*(\sigma) > 0$ and $\rho^*(\sigma,\sigma') > 0$ for all $\sigma\in \Sigma^m$.

In particular, $\rho^*$ is an \textit{irreducible} stochastic matrix and so by standard results in Markov theory \autocite[Theorem 6.4.4, 6.5.6]{durrett2010probability} there exists a unique stationary distribution $\eta_{\rho^*}$ on $\Sigma^m$ capturing the asympototic probability that state $\sigma\in \Sigma^m$ is observed on the $k^{\text{th}}$ trial, and since expected return times are finite for every irreducible Markov chain, $\eta_{\rho^*}(\sigma) > 0$. Thus, in the long run, with probability one relative to $(\mu,\rho)$ the sentence 
\[ \left( \bigwedge\limits_{i} \varphi_{\MM_i}(\overline{x}_i,s_i) \right) \wedge \bigwedge\limits_{i} (s_i < s_{i+1}) \]
is realized.

%Note that we may regard such a formula as an event, and by the assumptions on the stochastic matrix we have that 
%\[ \PP\left[ \left( \bigwedge\limits_{i} \varphi_{\MM_i}(\overline{x}_i,s_i) \right) \wedge \bigwedge\limits_{i} (s_i < s_{i+1}))\right] > 0\]
%and so by the memoryless property and the fact that we have $\omega$ many times, with probability 1 this formula is realized. Hence every $\exists_1$ $\mathcal{L}_{\tau}^d$-sentence that is consistent with $T_{\tau}^d$ is a member of $T_{\mathcal{C}_+}$, so 
%$T_{\mathcal{C}_+}$ is unfalsifiable relative to $T_{\tau}^d$. 

In other words, on such a model of the evolution of the universe, every consistent configuration of atomic sentences is realized with probability one according to this process. Hence, this theory is unfalsifiable.

The theories $T_{\mathcal{C}}$ naturally occur as a formal model of the universe as a theormodynamic fluctuation. The idea that the universe is merely a fluctuation has been discarded by many prominent physicists such as Feynman and Carroll; it is worth investigating how these arguments dovetail with the present discussion of their falsifiability.

Feynman argues that we can refute this hypothesis, writing:
\begin{quote} 
Thus one possible explanation of the high degree of order in the present-day world is that it is just a question of luck. Perhaps our universe happened to have had a fluctuation of some kind in the past, in which things got somewhat separated, and now they are running back together again...

[F]rom the hypothesis that the world is a fluctuation, all of the predictions are that if we look at a part of the world we have never seen before, we will find it mixed up, and not like the piece we just looked at. If our order were due to a fluctuation, we would not expect order anywhere but where we have just noticed it...
 is 
Every day they turn their telescopes to other stars, and the new stars are doing the same thing as the other stars. We therefore conclude that the universe is not a fluctuation. \autocite[Lecture 46-5]{feynman2011feynman}
\end{quote}

On this account, from the fact that we observe order---the aggregate of all of our observations of the universe---we can conclude that the universe is not a fluctuation.

At first glance this argument appears to be an argument from falsification:

\[\begin{nd}
\hypo {1}{\text{The Fluctuation Hypothesis entails that the universe is disordered.}}
\hypo {2}{\text{We observe order in the universe.}} 
\have {3}{\text{The Fluctuation Hypothesis is false.}} 
\end{nd}
\]

After all, it appears to be framed as a \textit{reductio ad absurdum}, but the inference is more subtle than that. If by the fluctuation hypothesis we understand it to mean that the universe is generated probabilistically in the manner described above, then observing order of arbitrarily large complexity is in fact a deductive consequence of the theory $T_{\mathcal{C}}$.

The tension here comes from a quirk of the probabilistic framework and its relation to first-order logic; while the probability of a \textit{specific} observer witnessing a given highly-ordered conjunction of atomic and negations of atomics formulas will be quite low,  nevertheless the theory \textit{predicts} that all such observations will be witnessed. In other words, two notions of \textit{prediction} are at play: in one sense, the theory \textit{entails} that with probability 1 the state that is observed will happen, all the while entailing that the observer in question witnesses a sequence of \textit{low probability states}. Carroll \autocite{carroll2020boltzmann} refers to this latter property of the fluctation theory as rendering observers ``cognitively unstable'' in the sense that the theory in question actively thwarts inductive reasoning as understood by Bayesian confirmation theory.

What Feynman has in mind, most likely, is an anthropic principle of the kind that says we should only affirm/consider theories $T$ which themselves make it highly probably that \textit{our own} inductive reasoning is highly conducive to truth.

Much ink has been spilled over anthropic principles in connection with the hypothesis that the universe is in some manner random \autocite{bostrom2013anthropic}, but the results of this section indicate that such theories suffer the defect of unfalsifiability. While being unfalsifiable does not refute the \textit{truth} of the hypothesis, it does show that the hypothesis is not amenable to being refuted by way of finitary modes of data acquisition.

\section{How Much of a Theory is Falsifiable?} 

The static models of falsifiability typically concern themselves with questions of how close a theory is to being universally axiomatizable; after all, the more universal sentences a theory implies, in principle the more falsfiable the theory becomes. 

Simon and Groen \autocite{simon1979fit}, in their work on Ramsification and the Second-Order definability of theories, isolate a notion on \textit{pseudoelementary} classes $\K$ they call FITness which they claim isolates the ideal scientific theories: On their account, a pseudoelementary class $\K$ is FIT if and only if it is a scientific theory. They show that for pseudoelementary $\K$, being FIT implies its universal axiomatizability.\footnote{The proof presented in their paper is incorrect. However, this error is fixed in this disseration.} Generalizing their definition to arbitrary classes of $\LL$-structures $\K$ closed under isomorphism, I show that for finite languages $\K$ being FIT entails that $\K$ is elementary and, in fact, universally axiomatizable. This result substantially generalizes their result over finite languages. 

I then turn my attention to an argument given by Chambers et al.~\autocite{chambers2014axiomatic} that argues that being universally axiomatizable is not sufficient grounds to call a theory falsifiable. Instead, they identify the falsifiable sentences with a class of universal sentences they call UNCAF (a \textul{u}niversal \textul{n}egation of a \textul{c}onjunction of \textul{a}tomic \textul{f}ormulas). In turn, I argue that their argument implicitly assumes that the underlying predicates $P\in \LL$ exhibit mere $\Sigma_1$ behavior  and thus that their argument reaches too far in its conclusions.

As a final foray into the static case of falsification, I consider how falsification intersects with the dividing lines of classification theory. It is not too difficult to show that under very mild restrictions on the language $\LL$, NIP theories entail a great deal of nontrivial $\forall_1$ sentences and are highly falsifiable. Of note, in NIP theories each formula $\varphi$ is equipped with a notion of dimension known as the VC-dimension of a class, which in a sense measures the effective falsifiability of membership in the class of hypotheses it defines. 

On the other hand, recent work of Kruckman and Ramsey \autocite{kruckman2018generic} and, independently, Je{\v{r}}{\'a}bek \autocite{jevrabek2020recursive}, yield examples of NSOP$_1$ and simple theories which are \textit{un}falsifiable. While there are many NSOP$_1$ theories which are falsifiable, in a sense NIP is individuated among the dividing lines in model theory as a class of highly-falsifiable theories.

\subsection{FITness: The Finite Signature Case}

We begin our investigation of the static case of falsification by exploring the notion of FITness---the \textit{finite} and \textit{irrevocable}---testability of a theory. Simon and Groen \autocite{simon1979fit} argue that, at least when $\K$ is a pseudoelementary class, $\K$ being FIT is necessary and sufficient for $\K$ to be a scientific theory. They purport to show that if  $\K$ is FIT and pseudoelementary, then $\K$ is $\forall_1$ axiomatizable. In this section I show that so long as the signature $\LL$ is finite, the requirement that $\K$ is pseudoelementary is unneccessary; all that is needed is that $\K$ is closed under $\LL$-isomorphism.

\begin{defn} \label{def:fitness}
Let $\LL$ be a language and $\mathbb{K}$ a class of $\LL$-structures. $\mathbb{K}$ is said to be \textit{FIT} provided that
\begin{enumerate}[label=\roman*]
\item $\mathbb{K}$ is \textit{finitely testable}, i.e., $\mathbb{K}$ is nontrivial:
\begin{equation*}
\mathbb{K} \neq \text{Str}(\LL)
\end{equation*}
and  for every $\MM\in \text{Str}(\LL)$,
\begin{equation*}
(\forall \NN\in\text{Str}(\LL) [(|\NN|<\aleph_0 \wedge \NN \subseteq_{\LL} \MM) \rightarrow \NN \in \mathbb{K}]) \rightarrow \MM \in \mathbb{K},
\end{equation*}
\item and $\mathbb{K}$ is \textit{irrevocably testable}, i.e., for every $\MM\in \text{Str}(\LL)$ 
\begin{equation*}
\MM \in \mathbb{K} \rightarrow (\forall \NN\in\text{Str}(\LL) [(|\NN|<\aleph_0 \wedge \NN \subseteq_{\LL} \MM) \rightarrow \NN \in \mathbb{K}]). \qedhere
\end{equation*}
\end{enumerate}
\end{defn}

In the case of a finite relational language $\mathcal{L}$, any FIT class $\mathbb{K}$ is universally axiomatizable. This substantially weakens the assumption on $\mathbb{K}$ given in the original paper of Simon and Groen at the cost of working within a more limited class of languages. 

\begin{thm}\label{thm:fitness_ax}
Let $\mathbb{K}$ be a FIT class of a structures over a finite relational language $\mathcal{L}$ closed under isomorphism. Then $\mathbb{K}$ is universally axiomatizable.
\end{thm}

\begin{proof}
We begin by giving a first-order axiomatization of $\mathbb{K}$. For each finite $\NN \in \mathbb{K}$, let $\varphi_{\NN}$ be the formula in $|\NN|$ many free variables given by $\bigwedge\limits_{\varphi \in diag(\NN)} \varphi$. This formula expresses the isomorphism type of $\NN$ relative to the fixed enumeration $x_1,\dots, x_n$: if $\NN \simeq_{\LL(\overline{x})} \NN'$ then $\varphi_{\NN}$ is equivalent to $\varphi_{\NN'}$.
Let $\mathbb{K}[n] = \{\NN\in \mathbb{K} \,|\, |\NN| \leq n \}/\simeq_{\LL(\overline{x})}$.
Since the language $\mathcal{L}(\overline{x})$ is finite and only includes relations and constant symbols, for each $n\in \omega$ there are only finitely $\mathcal{L}(\overline{x})$-isomorphism classes in $\mathbb{K}$ of size $\leq n$, so $\mathbb{K}[n]$ is finite. Let $\psi_n$ be the sentence
\[\psi_n = \forall x_1,\dots \forall x_n \left(\bigwedge\limits_{i\neq j} x_i\neq x_j \rightarrow \bigvee_{[\MM]\in \mathbb{K}[n]} \varphi_{\MM}\right).\]
By construction, each $\psi_n$ is a universal sentence, as the disjunction $\left(\bigvee_{[\MM]\in \mathbb{K}[n]} \psi_{\MM}\right)$ is a disjunction of finitely many Boolean combinations of atomic formulas.

Let $T_{\mathbb{K}} = \{\psi_n\}_{n\in\omega}$. I claim that $\mathbb{K} = \operatorname{Mod}(T_{\mathbb{K}})$. To see this, suppose that $M \models T_{\mathbb{K}}$. Because $\mathbb{K}$ is finitely testable it suffices to show that every finite substructure of $\MM$ is a member of $\mathbb{K}$. Let $\NN$ be a substructure of $\MM$ of size $n$. Since $\psi_n$ is universal, $N \models \psi_n$ and so $\NN \models  \psi_{\NN'}$ for some $\NN'$ isomorphic to a member of $\mathbb{K}$. Since $\mathbb{K}$ is closed under isomorphism, $\NN \in \mathbb{K}$. Thus $\MM \in \mathbb{K}$.

Conversely, suppose that $\MM\in \mathbb{K}$. To show that $\MM \models T_{\mathbb{K}}$, it suffices to show that $\MM \models \psi_n$ for each $n$. Let $(m_1,\dots,m_n) \in \MM^n$ be a variable assignment. The set $\NN = \{m_1,\dots,m_n\}$ is a set of size $\leq n$ and is a substructure of $M$. By the irrevocable testability of $\mathbb{K}$, $N \in \mathbb{K}$. Thus, $\NN\models \varphi_{[\NN]}$. Thus $\MM\models \psi_n$.
\end{proof}

Moreover, a similar argument works to show that a FIT class closed under isomorphism over an arbitrary finite language is universally axiomatizable.

\begin{thm} \label{-finitelanguage}
Let $\K$ be a FIT class of structures over a finite language $\LL$ closed under isomorphism. Then $\K$ is universally axiomatizable.
\end{thm}
\begin{proof}
Same as the above, but by defining the axiom scheme $\psi_n$ as follows. For a function symbol $f$, we denote the arity of $f$ by $ar(f)$.

Let $\chi_n(x_1,\dots,x_n)$ be the formula given by
\[\chi_n = \left( \bigwedge\limits_{f\in \LL}\left[\bigwedge\limits_{I\subset [n]^{ar(f)}}\bigvee\limits_{0<j<n} f(\overline{x}_I) = x_j\right] \wedge \bigwedge\limits_{c\in \LL}\bigvee\limits_{0<i\leq n} x_i = c \right),\]
where $\bigwedge\limits_{f\in \LL}$ and $\bigwedge\limits_{c\in \LL}$ are understood to be $\top$ in case $\LL$ contains no function or constant symbols respectively.

This formula expresses that the set $x_1,\dots, x_n$ is an $\LL$-structure of size $\leq n$, as it expresses that $x_1,\dots, x_n$ is  closed under all function symbols $f\in \LL$ and contains all constants $c\in \LL$. Since $\LL$ is finite, this is a quantifier-free first-order formula.

As above, let $T_{\K}$ be axiomatized by 
\[\psi_n = \forall x_1,\dots,x_n \left(\chi_n \rightarrow \bigvee\limits_{[\MM] \in \K[n]} \varphi_{\MM} \right). \]

A nearly identical argument as before suffices to show that $\K$ is axiomatized by $T_{\K}$.
Let $T_{\mathbb{K}} = \{\psi_n\}_{n\in\omega}$. I claim that $\mathbb{K} = \operatorname{Mod}(T_{\mathbb{K}})$. To see this, suppose that $M \models T_{\mathbb{K}}$. Because $\mathbb{K}$ is finitely testable it suffices to show that every finite substructure of $\MM$ is a member of $\mathbb{K}$. Let $\NN$ be a substructure of $\MM$ of size $n$. Since $\psi_n$ is universal, $\NN \models \psi_n$. Since $\NN$ is an $\LL$-structure of size at most $n$, $\NN \models \chi_n$, so $\NN \models \psi_{\NN'}$ for some $\NN'$ isomorphic to a member of $\mathbb{K}$. Since $\mathbb{K}$ is closed under isomorphism, $\NN \in \mathbb{K}$. Thus $\MM \in \mathbb{K}$.

Conversely, suppose that $\MM\in \mathbb{K}$. To show that $\MM \models T_{\mathbb{K}}$, it suffices to show that $\MM \models \psi_n$ for each $n$. Let $(m_1,\dots,m_n) \in \MM^n$ be a variable assignment. The set $\NN = \{m_1,\dots,m_n\}$ is a set of size $\leq n$. If $\NN$ is an $\LL$-structure, then $\NN$ is a substructure of $\MM\in \K$ so $\NN \models \varphi_{[\NN]}$ and hence \[\MM \models \chi_n(m_1,\dots,m_n) \rightarrow \bigvee\limits_{[\MM] \in \K[n]} \varphi_{\MM}(m_1,\dots,m_n).\] If $\NN$ is not a substructure, then the variable assignment satisfies $\MM \models \neg \chi_n(m_1,\dots,m_n)$ and so \[\MM \models \chi_n(m_1,\dots,m_n) \rightarrow \bigvee\limits_{[\MM] \in \K[n]} \varphi_{\MM}.\]
Thus $\MM \models \psi_n$ for all $n$, so $\MM \models T_{\K}$.
\end{proof}

\begin{propo}
Suppose that $T$ is a universally axiomatizable class over a finite signature $\mathcal{L}$. 

\begin{enumerate}
    \item If $\LL$ is relational, then $\operatorname{Mod}(T)$ is a FIT class.
    \item There exist finitely axiomatizable $T$ which are not FIT. \qedhere
\end{enumerate}

\end{propo}
\begin{proof}
Suppose that $\LL$ is relational. We need to show that for all $\LL$-structures $\MM$, $\MM \models T$ if and only if $\NN \models T$ for all finite $\NN\subset \MM$.

Suppose that $\MM \models T$. Then since $T$ is universally axiomatizable, $\NN \models T$ for all substructures $\NN\subset \MM$. On the other hand, suppose $\MM \not\models T$. Then there exists a $\forall_1$ sentence $\varphi\in T$ such that $\MM\models \neg \varphi$. Since $\varphi$ is $\forall_1$, $\neg\varphi$ is $\exists_1$, there exists a witness $\overline{m}$ to $\MM \models \neg \varphi$. The finitely generated substructure $\MM_0 = \left\langle m\right\rangle \subset \MM$ also satisfies $\MM_0 \models \neg\varphi$. Since $\LL$ is relational, $\MM_0$ is a finite structure, so $\MM\not\models T$ ensures that there is a finite $\MM_0\subset \MM$ with $\MM_0\not\models T$. Thus the models of $T$ form a FIT class.

On the other hand, let $\LL = \{f(x),g(x),c\}$ and let $T$ be the theory given by the single universal axiom: $(\forall x)\; g(x) \neq x.$ We now exhibit an example of an $\LL$-structure $\MM$ such that $\MM\not\models T$ but $\NN \models T$ for all finite substructures $\NN\subset \MM$. Let $\MM$ have domain $\omega$ and interpret $f(x) = S(x)$ the successor function, $g(x) = x$ the identity function, and $c=0$. Note that $\MM$ has no finite substructures, so vacuously $\NN\models T$ for all finite structures $\NN\subset \MM$. However, $\MM\not\models T$, so $\MM$ is not FIT.
\end{proof}

\subsection{FITness: The Arbitrary Language Case} 

The setting in which Simon and Gr{\"o}en work involves a distinction between \textit{observational} and \textit{theoretical} scientific terms. Let $\LL = \LL_o \cup \LL_t$ be a language partitioned into the \textit{observational} language $\LL_o$ and \textit{theoretical language} $\LL_t$. Let $\Sigma$ be an $\LL$-theory. There is, of course, the class of models of the theory:
\begin{equation*} \text{Mod}(\Sigma) = \{ \MM\,|\, \MM\models \Sigma \} \subset \text{Str}(\LL).
\end{equation*} 
By definition, this class is \textit{elementary}, meaning that it is first-order axiomatizable. However, the class $\text{Mod}(\Sigma)$ is \textit{not} the appropriate class of structures to look at, for if there is a true \textit{o/t} distinction then the scientist only has epistemic access to the \textit{observable} structure. Instead, Sneed \autocite{sneed1979physics} isolates the fundamental relation between scientific $\LL$-theory $\Sigma$ and some $\LL_o$-structure $\NN$ of observations is that of \textit{application}: say that $\Sigma$ \textit{applies to} $\NN$ just in case the $\LL_o$-structure $\NN$ can be expanded\footnote{An \textit{expansion} of an $\LL_o$ structure $\NN$ to an $\LL$-structure $\widetilde{N}$ is an $\LL$-structure where the domain of $\widetilde{\NN}$ is $\NN$ and all of the symbols in $\LL_o$ are interpreted as is $\NN$.} to a full $\LL$-structure $\widetilde{N}$ such that $\widetilde{N}\models \Sigma$. The \textit{pseudoelementary} class of such structures is given by:
\begin{equation*}
\text{Mod}^{*}(\Sigma) = \{ \MM|_{\LL_o} \,|\, \MM\models \Sigma \} \subset \text{Str}(\LL_o).
\end{equation*} 

In the case of pseudoelementary classes $\mathbb{K} = \text{Mod}^{*}(\Sigma)$, we are able to drop the hypothesis that $\LL$ is a finite language to conclude that an $\LL_o$-FIT $\K$ is $\forall_1$-axiomatizable. This is the original result of Simon and Groen \autocite{simon1979fit}.

\begin{propo}
Let $\Sigma$ be an $\LL_o$-\textit{FIT} theory. Then the class $\operatorname{Mod}^{*}(\Sigma)$ is an elementary class and admits a universal axiomatization.\footnote{In \autocite{simon1979fit} Simon claims that this result follows from the {\L{}}o\'s-Tarski theorem. However, the {\L}o\'s-Tarski theorem applies to elementary classes, whereas the class in question---$\operatorname{Mod}^{*}(\Sigma)$---is a pseudoelementary class. Thus a new proof is needed.}
\end{propo}

\begin{proof} We recall a theorem of model theory \autocite[Theorem~6.6.7]{hodges1993model}:
\begin{quote} \textit{Let $\LL$ be a first-order language and $\K$ be a pseudo-elementary class of $\LL$-structures. Suppose that $\K$ is closed under taking substructures. Then $\K$ is axiomatized by a set of $\forall_1$ $\LL$-sentences. }
\end{quote}
Since $\operatorname{Mod}^{*}(\Sigma)$ is pseudo-elementary, it suffices to show that $\operatorname{Mod}^{*}(\Sigma)$ is closed under substructures. Let $\MM \in \operatorname{Mod}^{*}(\Sigma)$ and let $\NN\subset_{\LL_o} \MM$ be a substructure. To show that $\NN\in \operatorname{Mod}^{*}(\Sigma)$, the \textit{finite testability} implied by $\LL_o$-\textit{FIT}-ness tells us that we need only check that for every finite substructure $\NN_k \subset_{\LL_o} \NN$ satisfies $\NN_k \in \operatorname{Mod}^{*}(\Sigma)$. Since every such $\NN_k$ is an $\LL_o$-substructure of $\MM$, the \textit{irrevocability} of $\LL_o$-\textit{FIT}-ness ensures that $\NN_k \in \operatorname{Mod}^{*}(\Sigma)$.
\end{proof}

That is, \textit{FIT}-ness implies that the pseudo-elementary class of $\LL_o$-structures expandable to models of $\Sigma$ is not only \textit{elementary}, but is in fact axiomatizable by \textit{universal} axioms. 

A partial converse can be given for the case of relational observational languages $\LL_o$:

\begin{propo}
Suppose $\operatorname{Mod}^{*}(\Sigma)$ is a universally axiomatized class of $\LL_{o}$-structures, axiomatized by $T^{*}_{\Sigma}$, such that
\begin{enumerate}[label=\roman*]
\item $\Sigma$ has nontrivial observational consequences, i.e., 
\begin{equation*} \operatorname{Mod}^{*}(\Sigma) \neq \text{Str}(\LL_o), 
\end{equation*}
 and
\item $\LL_{o}$ is a relational language.
\end{enumerate} 
Then $\Sigma$ is $\LL_o$-\textit{FIT}.
\end{propo}

\begin{proof}
To show that $\Sigma$ is $\LL_o$-\textit{FIT} we must show both \textit{finite testability} and \textit{irrevocable testability}.

\textbf{Finite testability:} Since $\operatorname{Mod}^{*}(\Sigma) \neq \text{Str}(\LL_o)$, it follows that that $\operatorname{Mod}(\Sigma) \neq \text{Str}(\LL)$, for otherwise \textit{all} $\LL_o$ structures would be reducts of models of $\Sigma$.

We now need to show that if, for all $\MM_k \subset_{\LL_o} \MM$ finite, $\MM_k \in \operatorname{Mod}^{*}(\Sigma)$, then $\MM \in \operatorname{Mod}^{*}(\Sigma)$. It is a known result \autocite[Exercise 2.5.20]{marker2006model} that any structure $\MM$ is $\LL_o$-embeddable\footnote{Not necessarily elementarily.} into an ultraproduct of its finitely-generated substructures. Since $\LL_o$ is relational, the finitely-generated substructures are precisely the \textit{finite} substructures. Thus, there is an $\LL_o$-embedding $\iota: \MM \hookrightarrow \prod\limits_{\mathcal{U}} \MM_k$ for $\mathcal{U}$ any nonprincipal ultrafilter over the collection of finite substructures of $\MM$. Now, as $\MM_k \models T^*_{\Sigma}$ for all finite $\MM_k \subset_{\LL_o} \MM$, $\prod\limits_{\mathcal{U}} \MM_k \models T^{*}_{\Sigma}$. Since $T^{*}_{\Sigma}$ is universally axiomatizable and $\MM \subset_{\LL_o} \prod\limits_{\mathcal{U}} \MM_k$, $\MM \models T^{*}_{\Sigma}$. But this means that $\MM \in \operatorname{Mod}^*(\Sigma)$. So $\Sigma$ is finitely testable.

\textbf{Irrevocable testability:} We need to show that if $\MM \in \operatorname{Mod}^{*}(\Sigma)$ then for all $\MM_k \subset_{\LL_o} \MM$ finite, $\MM_k \in \operatorname{Mod}^{*}(\Sigma)$. Since $T^*_{\Sigma}$ is universally axiomatizable, \textit{any} $\LL_o$-substructure of $\MM$ is a model of $T^*_{\Sigma}$. In particular, each finite $\MM_k \subset_{\LL_o} \MM$ is a model of $T^*_{\Sigma}$ and is therefore a member of $\operatorname{Mod}^{*}(\Sigma)$, as desired.

Hence $\Sigma$ is $\LL_o$-\textit{FIT}.
\end{proof}

The two properties defining \textit{FIT}-ness warrant scrutiny in virtue of their strong implications. We may view the finitely testable hypothesis as a \textit{local compactness} principle: in the stated form it says that if every finite $\MM_k \subset \MM$ is \textit{consistently expandable to a model of $\Sigma$}, so too is $\MM$. The irrevocability hypothesis expresses the closure of the class $\operatorname{Mod}^{*}(\Sigma)$ under (finite) substructures, which together with finite testability implies closure under substructures.

Moreover, when working with a relational language the semantic criterion of \textit{FIT}-ness is equivalent to the universal axiomatizability of the \textit{observable} consequences of the theory. Thus, on the Simon-Groen view, given a universally axiomatizable $\LL_o$ theory $T$ \textit{any} $\LL_o$-conservative extension of $T$ to an $\LL$-theory $T'$ is scientific. For instance, consider adding unary predicate symbols $P_1,\dots, P_n$ and defining 
\[T_m = T\cup \{\forall x \left( \bigvee\limits_{1\leq i\leq n} P_i(x) \right) \}.\]
the $\LL_o$-consequences of $T_m$ are the $\LL_o$-consequences of $T$ and so $T_m$ is \textit{FIT} and therefore scientific. By construction, however, the truth of the axioms of $T_m$ are independent from any collection of observational data.

\subsection{FITness and Finite Generation}

The definition of FITness required that membership in a class $\K$ be witnessed by all finite substructures themselves being members of $\K$. However, except in the relational case, a substructure being finitely generated does not imply that that substructure is finite. In this section we consider the analogous notion of FITness obtained by replacing ``finite'' with ``finitely generated'' everywhere in the definition of FITness.

\begin{defn} \label{def:fg-fitness}
Let $\LL$ be a language and $\mathbb{K}$ a class of $\LL$-structures. $\mathbb{K}$ is said to be \textit{fg-FIT} provided that
\begin{enumerate}[label=\roman*]
\item $\mathbb{K}$ is \textit{fg testable}, i.e. $\mathbb{K}$ is nontrivial:
\begin{equation*}
\mathbb{K} \neq \text{Str}(\LL)
\end{equation*}
and that for every $\MM\in \text{Str}(\LL)$,
\begin{equation*}
(\forall \NN\in\text{Str}(\LL) [(\NN \text{ is finitely-generated}  \wedge \NN \subseteq_{\LL} \MM) \rightarrow \NN \in \mathbb{K}]) \rightarrow \MM \in \mathbb{K},
\end{equation*}
\item and $\mathbb{K}$ is \textit{fg-irrevocably testable}, i.e. for every $\MM\in \text{Str}(\LL)$ 
\begin{equation*}
\MM \in \mathbb{K} \rightarrow (\forall \NN\in\text{Str}(\LL) [(\NN\text{ is finitely-generated} \wedge \NN \subseteq_{\LL} \MM) \rightarrow \NN \in \mathbb{K}]). \qedhere
\end{equation*}
\end{enumerate} 
\end{defn}

The FITness and fg-FITness of a class are generally inequivalent.

\begin{propo}
Let $\mathcal{L} = \{+,-,\times,0,1\}$ be the language of rings, and let $\mathbb{K}$ be the class of all rings of positive characteristic, i.e. rings such that there exists an $n \in \omega$ such that $\underbrace{1 + \cdots + 1}_{n\rm\ times} = 0$. Then $\mathbb{K}$ is fg-FIT but not first-order axiomatizable. In particular, $\K$ is not FIT.
\end{propo}

\begin{proof}
To show that $\mathbb{K}$ is fg-FIT, it suffices to show that for a ring $R$, $R\in \mathbb{K}$ just in case every finitely-generated subring of $R$ is in $\mathbb{K}$. Suppose that $R\in \mathbb{K}$. This is witnessed by the quantifier-free formula $\underbrace{1 + \cdots + 1}_{n\rm\ times} = 0$, so any subring $R'\subset R$ is also a member of $\K$. Conversely, if $R\notin \K$ then $\left\langle 1\right\rangle$ is infinite and therefore $\left\langle 1\right\rangle \notin \K$.

To show that $\K$ is not first-order axiomatizable, it suffices to show that $\K$ is not closed under ultraproducts by \autocite[Theorem 4.1.12]{changkeisler}. Note that each finite field $F_p$ is a member of $\K$. Let $\mathcal{U}$ be a nonprincipal ultrafilter on the set of primes. Then $F = \prod\limits_{\mathcal{U}} F_p$ is a field of characteristic zero, thus $F\notin \K$. Hence $\K$ is not first-order axiomatizable. Therefore, by \autoref{thm:fitness_ax}, $\K$ is not FIT.
\end{proof}

Moreover, unlike the FIT case, every universally axiomatizable theory is fg-FIT.

\begin{propo}
Suppose that $T$ is a universally axiomatizable class over an arbitrary signature $\LL$. Then $T$ is fg-FIT.
\end{propo}
\begin{proof}
We need to show that for all $\LL$-structures $\MM$, $\MM \models T$ if and only if $\NN \models T$ for all finitely-generated $\NN\subset \MM$.

Suppose that $\MM \models T$. Then since $T$ is universally axiomatizable, $\NN \models T$ for all substructures $\NN\subset \MM$. On the other hand, suppose $\MM \not\models T$. Then there exists an $\forall_1$ sentence $\varphi\in T$ such that $\MM\models \neg \varphi$. Since $\varphi$ is $\forall_1$, $\neg\varphi$ is $\exists_1$, there exists a witness $\overline{m}\in \MM^k$ to $\MM \models \neg \varphi$. The finitely generated substructure $\MM_0 = \left\langle m\right\rangle \subset \MM$ also satisfies $\MM_0 \models \neg\varphi$. Thus $\MM\not\models T$ ensures that there is a finite $\MM_0\subset \MM$ with $\MM_0\not\models T$. Thus the models of $T$ form an fg-FIT class.
\end{proof}

\subsection{Remarks on Signatures in FITness}

In the above discussions regarding FITness, fg-FITness, and universal axiomatizability, it was shown that in the case of a finite relational language, these notions are equivalent without a background assumption on the class $\K$ beyond closure under $\LL$-isomorphism. However, these notions began to decouple in the case of languages with constant symbols and function symbols. This behavior is not so surprising; when converting a function symbol $f$ to a relation symbol by defining $R_f$ by identifying $\forall x\forall y R_f(x,y) \leftrightarrow f(x)=y$, to eliminate the function symbol $f$ from the language completely requires one to include an  $\forall_2$ axiom of the form
\[ \forall x \exists y R_f(x,y), \]
which in general will \textit{not} be equivalent to a $\forall_1$ sentence. Thus, implicit in the inclusion of function symbols in the language is a $\forall_2$ axiom in a purely relational language.

\subsection{UNCAF Theories}

Motivated by theories of revealed preference in economics, Chambers et al. \autocite{chambers2014axiomatic} argue that the empirical content of a theory is captured not by general universal sentences but instead by a special kind of universal sentence they term UNCAF.

\begin{defn}\autocite[Definition 4]{chambers2014axiomatic}
An UNCAF sentence in a language $\LL$ is a \textul{u}niversal \textul{n}egation of a \textul{c}onjunction of \textul{a}tomic \textul{f}ormulas; that is, a sentence of the form
\[(\forall x_1,\dots x_n) \neg\left(\bigwedge\limits_{1\leq i < m} \varphi_i(x_1,\dots,x_n) \right) \]
where each $\varphi_i$ is atomic.
\end{defn}

Perhaps surprisingly, they argue that sentences of the form $(\forall x) P(x)$ is \textit{not} falsifiable by virtue of not being UNCAF, while $(\forall x)\neg P(x)$ is.

To argue this point, they write that
\begin{quote}
    substructures are unsatisfactory as mathematical models for observed data since they correspond to a situation in which the scientist observes the presence or absence of every possible relation among the elements in his data and, therefore, cannot accommodate partial observability.
\end{quote}

While I agree with this general point, the conclusion that only UNCAF sentences have empirical content is too strong. For example, let $S(x)$ be the predicate ``$x$ is a swan'' and $W(x)$ the predicate ``$x$ is white.'' The sentence ``all swans are white,'' when formalized, is equivalent to
\[ (\forall x)\neg (S(x) \wedge \neg W(x)),\]
which is not UNCAF owing to the presence of $\neg W(x)$ as a nested subformula. To conclude that this sentence has no falsificational content seems to run counter to the usual conception of falsification: after all, \textit{if} I were able to produce an example $c$ such that $S(c) \wedge \neg W(c)$, I would immediately be able to infer that $W\notin \mathbb{K}$. However, on their model it is as if, when I go to the local bird sanctuary I am told that I may \textit{only} record instances of white swans. One should not expect to be able to produce a counterexample to ``all swans are white'' under such constraints! 

The way that Chambers et al. circumvent this worry is to note that for each predicate $P$ one may add a new relation symbol $P^{\neg}$ together with the axiom
\[\forall x (\neg P(x) \iff P^{\neg}(x)).\]
While this approach does formally work, it is somewhat awkward that this axiom itself is \textit{not} UNCAF, as we see by reducing it to
\[\forall x\neg((\neg P(x) \wedge \neg P^{\neg}(x)) \vee (P(x) \wedge \neg P^{\neg}(x))).\]

Their understanding of falsification \textit{qua} UNCAF-expressibility entangles two separate considerations: first, whether there is \textit{in principle} any falsificational strategy on the basis of some configuration being witnessed by a finite set of data, and \textit{second} whether the model of knowledge acquisition allows one to actually carry out the falsificational strategy. Their account corresponds to a model of knowledge acquisition in which at each stage one gains (at most) one \textbf{positive} (relative to $\mathcal{L}$) observation at a time, in a semidecidable fashion.

As an example, suppose that a researcher is observing an agent Ashley and wishes to falsify whether or not her preference relation is complete:
\[\forall x,y \left((x \leq y) \vee (y \leq x)) \right).\]
To do so, the observer waits each day $d$ to see whether the agent exhibits some preference relation between a can of Guayak{\'i} Enlighten Mint ready-to-drink Yerba Mate and a can of Guayak{\'i} Revel Berry that are sitting side-by-side in the office fridge, with no other items in potential consideration.

This experiment, as construed, is doomed to never falsify the experiment. After all, if there is some day $d$ where Ashley surveys the fridge and takes a can of Enlighten Mint but not Revel Berry (resp. Revel Berry but not Enlighten Mint), then $EM \geq RB$ (resp. $RB \geq EM$) and therefore no refutation of the completeness axiom is possible in the context. Likewise, if the day that Ashley takes a can out of the fridge never comes, that \textit{also} does nothing to falsify the completeness axiom.

So, what went wrong? Implicit in their semantics for the experiment is a suppressed existentially-defined quantifier. Let $R_A(x,y,d)$ be the relation that says ``on day $d$, agent $A$ expressed a weak preference $x$ over $y$.'' Then the formula $x\geq y$ in Chambers' terminology would not be $\forall_1$ but instead properly $\forall_2$:
\[ \forall x,y \left((\exists t)R_A(x,y,t) \vee (\exists t)R_A(y,x,t)\right).\]
Therefore, the purported example of an unfalsifiable $\forall_1$ sentence is better and more directly modeled as an unfalsifiable $\forall_2$ sentence fully compatible with the standard account of falsification as a universal over an in-principle decidable primitive. What their point indicates is that the standard revealed preference relations in economics are \textit{not} in-principle decidable, but instead are $\exists_1$-definable relative to the empirical relation $R_A(x,y,d)$ via
\[ x \geq y := (\exists t)R_A(x,y,d). \]

If we take as epistemically primitive a $\exists_1$-definable relation $R(x,y)$ defined by an $\LL$-formula $\exists c \varphi(x,y,c)$ with $\varphi$ quantifier-free, then their result is clear. A sentence of form $(\forall_1) R(x,y)$ is an $\exists_2$ sentence, while an UNCAF sentence in the language $\LL_R = \{R(x,y)\}$,
\[(\forall x_1,\dots,x_n) \neg\left(\bigwedge\limits_{i,j \in I} R(x_i,x_i) \right), \]
with $I\subset 2^{[n]}$ finite is equivalent to the $\LL$-sentence:
\[(\forall x_1,\dots,x_n) \left(\bigvee\limits_{i,j \in I} \forall c\; (\neg \varphi(x_i,y_i,c)) \right) \]
which is equivalent to a $\forall_1$ sentence in $\LL$.

\subsection{Strength of Theories and their Falsifiability}

Contrary to mere falsifiability, FITness, fg-FITness, and UNCAF-axiomatizable are typically not closed upwards under strength.

\begin{propo}
Let $\K$ be a universally axiomatizable FIT or UNCAF class such that $\K$ has both finite and infinite models. Then the class $\K'\subset \K$ of infinite members of $\K$ is not FIT, fg-FIT, or UNCAF.
\end{propo}
\begin{proof}
Let $T$ be a universal axiomatization of $\K$. Then $\K'$ is axiomatized by $T\cup \{\psi_n\}_{n\in\omega}$ where $\psi_n$ is the sentence $(\exists x_1,\dots x_n)\bigwedge\limits_{1\leq i\neq j\leq n} x_i\neq x_j$. 

Since $\K$ is closed under substructures and admits finite models, $\K'$ necessarily fails to be closed under substructures. Thus $\K'$ is not universally axiomatizable, and in particular is neither FIT nor UNCAF.
\end{proof}

\subsection{Falsification and NIP}

In the preceding sections, we have considered a sequence of refinements to the basic notion of falsifiability. We have seen, under mild conditions on the signature $\LL$ and classes $\K$ the web of implications% https://q.uiver.app/?q=WzAsNCxbMiwwLCJcXG1hdGhiYntLfSBcXHRleHR7IEZhbHNpZmlhYmxlfSJdLFsyLDEsIlxcbWF0aGJie0t9IFxcdGV4dHsgaGFzIG5vbnRyaXZpYWwgVU5DQUYgdGhlb3J5fSJdLFsxLDAsIlxcbWF0aGJie0t9XFw7XFw7IFxcZm9yYWxsXzEtXFx0ZXh0e0F4aW9tYXRpemFibGV9Il0sWzAsMCwiXFxtYXRoYmJ7S30gXFx0ZXh0eyBGSVR9Il0sWzIsMF0sWzEsMF0sWzMsMl1d
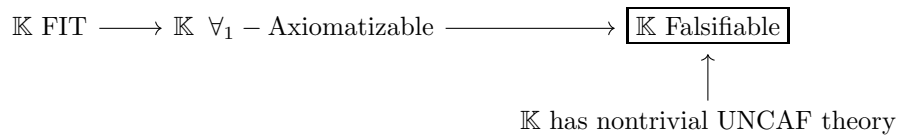
\[\begin{tikzcd}
	{\mathbb{K} \text{ FIT}} & {\mathbb{K}\;\; \forall_1-\text{Axiomatizable}} & {\boxed{\mathbb{K} \text{ Falsifiable}}} \\
	&& {\mathbb{K} \text{ has nontrivial UNCAF theory}}
	\arrow[from=1-2, to=1-3]
	\arrow[from=2-3, to=1-3]
	\arrow[from=1-1, to=1-2]
\end{tikzcd}\]

However, there are a great deal of hypotheses which do not readily fall into this framework at first glance.

For example, we are often interested in testing whether or not a (basic) relation $R(x_1,\dots,x_n)\in \LL$ is equivalent to some other (basic) relation $S(x_1,\dots,x_n)\in \LL$. This is easy to handle directly in our account of falsification; after all,
\[(\forall x_1,\dots,x_n) (R(x_1,\dots,x_n) \leftrightarrow S(x_1,\dots,x_n)) \]
is a $\forall_1$ sentence in $\LL$ by assumption.

What if, instead, we are probing a more complicated question, such as whether or not $R(x_1,x_2)$ is a line in $\RR^n$? In the language of rings augmented by an additional relation symbol
\[\LL = \{+,\times,0,1,<, R(x,y)\}\]
this is most easily expressed by the $\exists_2$ formula 
\[T_L = (\exists a,b)(\forall x,y) (R(x,y) \leftrightarrow L(x,y;a,b,c) ) \]
where $L(x,y;a,b)$ is the sentence $ay+bx+c=0$. Despite being $\exists_2$, this sentence has a great deal of falsificational content owing to the structure of the parametric family $L(x,y;a,b,c)$. From Euclidean geometry that between any two distinct points there exists a unique line. Letting $R^*(a,b,c,d)$ be the sentence \[R^*(a,b,c,d) = ([(a\neq c) \vee (b\neq d)] \wedge R(a,b) \wedge R(c,d)).\] Then
\begin{align*}
(\forall x,y)(\forall a,b,c,d) \left(R^*(a,b,c,d) \rightarrow \left(   L\left(x,y,\frac{d-b}{c-a},b-a\frac{d-b}{c-a}\right)\rightarrow R(x,y) \right)\right)
\end{align*}
which is a nontrivial $\forall_1$ sentence. Thus, while $T_L$ is $\exists_2$ it has nontrivial $\forall_1$ consequences.

These properties of lines is an example of the \textit{VC finiteness} of the class. For the remainder of the section, we assume that $\K$ is an \textit{elementary} class, axiomatized by some first-order set of sentences $T$.

\begin{defn}\autocite[pg. 7-8]{simon2015guide}
Let $\varphi(\overline{x};\overline{y})$ be a first-order formula in disjoint sets of free variables $\overline{x}, \overline{y}$. With respect to this partition we say that $\varphi$ is a \textit{partitioned} formula.

Let $\MM\in \K$. We say that $\varphi(\overline{x};\overline{y})$ $\MM$-shatters a set $X\subset \MM^{|\overline{x}|}$ just in case there is a set $Y\subset \MM^{|\overline{y}|}$ such that for every subset $X'\subset X$ there exists $y'\in Y$ such that
\[\MM \models \varphi(x;y') \iff x\in X'\]
for all $x\in X$.

A partitioned formula is \textit{NIP} provided for every $\MM \in \K$, no infinite set is $\MM$-shattered by $\varphi$. 

The formula $\varphi$ has Vapnik-Chervonenkis (VC) dimension, $VC(\varphi) \leq n$ just in case for all $\MM\in \K$, no set of size $n$ is $\MM$-shattered. If $\varphi$ has finite VC dimension then $\varphi$ is said to be \textit{VC finite}.

A theory $T$ is NIP just in case every formula $\varphi$ is NIP in the class $\K = \operatorname{Mod}(T)$. \end{defn}

For elementary classes $\K$, a formula being NIP is related to its VC finiteness:

\begin{propo}
Let $\K$ be an elementary class and $\varphi(\overline{x};\overline{y})$ a partitioned first-order formula. If $\varphi$ is NIP, then $\varphi$ has finite VC dimension.
\end{propo}
\begin{proof}
This is an elementary consequence of the compactness theorem of First-Order Logic \autocite[Remark 2.3]{simon2015guide}.
\end{proof}

A first-order formula $\varphi(\overline{x};\overline{y})$ having VC dimension $\leq n$ is first-order expressible by a sentence $VC_n(\varphi)$; moreover, if $\varphi$ is quantifier-free then the proposition $VC_n(\varphi)$ is a $\forall_1$ sentence.

\begin{propo}
A formula $\varphi(x;y)$ having VC dimension $\leq n$ is first-order expressible in any language containing $\varphi$ by a sentence $VC_n(\varphi)$.  Moreover, if $\varphi$ is $\exists_m / \forall_m$, then $VC_n(\varphi)$ is at most $\forall_{m+1}$.

In particular, if $\varphi$ is quantifier-free then $VC_n(\varphi)$ is a $\forall$ sentence.
\end{propo}
\begin{proof}
First, the proposition ``$x_1,\dots,x_n$ is shattered by $\{y_J \}_{J\subseteq [n]}$ in $\varphi$'' is a Boolean combination of instances in $\varphi$:
\[\operatorname{Shatter}_{\varphi}\left((x_i)_{1\leq i\leq n},(y_J )_{J\subseteq [n]}  \right) \iff \bigwedge\limits_{J\subseteq [n]} \bigwedge\limits_{1\leq i \leq n} \square_{i,J} \varphi(x_i,y_J) \]
where $\square_{i,J} \varphi$ is $\neg \varphi$ if $i\notin J$ and $\varphi$ if $i\in J$.

The proposition $VC_n(\varphi)$ is expressed by the following first-order sentence:
\[(\forall(x_i)_{1\leq i\leq n})\forall(y_J)_{J\subseteq [n]}\left( \left(\bigwedge\limits_{1\leq i\neq j\leq n} x_i\neq x_j\right)\rightarrow \neg \operatorname{Shatter}_{\varphi}\left((x_i)_{1\leq i\leq n},(y_J )_{J\subseteq [n]}  \right)  \right),\]
as desired.
\end{proof}

\begin{propo}
Let $T$ be a complete NIP theory in a language $\LL$ containing an $m$-ary relation symbol $R(\overline{x},\overline{y})$ for some $n>1$. Then $T$ implies a nontrivial universal sentence.
\end{propo}
\begin{proof}
Since $T$ is complete NIP, for some $T$ entails the $\forall_1$ sentence $VC_n(R(\overline{x};\overline{y}))$ for some $n\in\omega$. Since $R$ non-unary, $VC_n(R(\overline{x};\overline{y}))$ is not a first-order validity, since the bipartite graph $G_n$ on a disjoint set of vertices $[n]\cup 2^{[n]}$ given by $R(i,X) \iff i\in X$ satisfies
\[G_n\models \neg VC_n(R(\overline{x};\overline{y})).\qedhere \]
\end{proof}

In fact, since $VC_n(R(x;y)) \rightarrow VC_m(R(x;y))$ for all $m>n$, for a VC finite relation we get a nested chain of $\forall_1$ sentences. As we will see in our account of the dynamic case of falsification, this simple observation has very strong consequences in terms of understanding small-sample falsificational problems.

To explain the restriction about the language, we note that there are NIP unary theories entailing no nontrivial $\forall_1$ sentence. Recall \autocite[Definition 4.2.17]{marker2006model}\footnote{In the next section we will work with an alternative, equivalent definition of stability better-suited for our purposes.} that a theory $T$ is \textit{$\kappa$-stable} for a cardinal $\kappa$ if for every model $\MM\models T$, $n\in\omega$, and $A\subseteq \MM$ of size $\kappa$, the space of $n$-types with parameters in $A$ has size $\kappa$
\[ |S_n(A)| = \kappa.\]
A theory is stable provided it is $\kappa$-stable for some infinite $\kappa$. It is well known that stability implies NIP \autocite[Theorem 4.7]{shelah1990classification}.

\begin{propo}
There exists a stable theory $T$ in a unary language which entails no nontrivial universal sentence.
\end{propo}
\begin{proof}
Let $T$ be the theory in the language $\LL = \{P(x)\}$ axiomatized by 
\[\varphi_n = (\exists x_1,\dots, x_n)\left( \bigwedge\limits_{i\neq j < n} x_i\neq x_j \wedge \bigwedge\limits_{i<n} P(x_i)\right)\]
and 
\[\psi_n = (\exists x_1,\dots, x_n)\left( \bigwedge\limits_{i\neq j < n} x_i\neq x_j \wedge \bigwedge\limits_{i<n} \neg P(x_i)\right).\]
This theory is clearly $\aleph_0$-categorical: any countable model $\MM$ can be partitioned by
\[\MM = P(\MM) \cup \neg P(\MM) \]
with each definable set $P(\MM),\neg P(\MM)$ countably infinite. If $\MM,\NN \models T$, then any pair of bijections
\[f_P:P(\MM) \to P(\NN)\] 
and 
\[f_{\neg P}: \neg P(\MM) \to \neg P(\NN)\]
induce an $\LL$-isomorphism
\[f_P\cup f_{\neg P}: \MM \to \NN.\]
Moreover, $T$ has no finite models, so by Vaught's test \autocite[Theorem 2.2.6]{marker2006model} $T$ is complete. Clearly, $\forall_1(T)$ contains only first-order validities.

This theory is $\omega$-stable. Let $A$ be a set of size $\leq \aleph_0$. The types over $A$ are determined by specifying which coordinates $x_i$ are equal to an element of $A$ and, for those $x_i \notin A$, whether or not $P(x_i)$ or $\neg P(x_i)$. Thus, there are at most $(|A|+2)^n \leq \aleph_0$ types over $A$.
\end{proof}

Therefore, again under mild conditions on the language
\[\K \text{ NIP} \rightarrow \K \text{ Falsifiable}. \]

We observe that VC finiteness is not equivalent to universal axiomatizability.

\begin{propo}
There exists an NIP $T$ such that $T$ is not universally axiomatizable. There exists a universally axiomatizable $T$ such that $T$ is not NIP.
\end{propo}
\begin{proof}
Let $DLO$ be the theory of dense linear orders in the language $\LL = \{<\}$. Then $T$ is not universally axiomatizable as all of its models are infinite and the language is relational. Concretely, we know $\Q \models DLO$ but no finite subset $X\subset \Q$ is a model of $DLO$. Since $\LL$ is relational, $X$ is a substructure.

On the other hand, let $T$ be the (incomplete) theory of acyclic directed graphs in the language $R(x,y)$. $T$ is universally axiomatizable by the collection $\varphi_n$ of sentences defined by
\[\varphi_n = (\forall x_1,\dots,x_n) \neg\left( R(x_n,x_1) \wedge \bigwedge\limits_{1\leq i < n} R(x_i,x_{i+1})\right).\]
This class is not NIP as for each $n$ the bipartite digraph $G_n$ on a disjoint set of vertices $[n]\cup 2^{[n]}$ given by $R(i,X) \iff i\in X$ satisfies
\[G_n\models \neg VC_n(R(x;y))\]
and is a model of $T$.
\end{proof}
\subsection{Useful Examples of NIP Theories and VC finite Classes}

The preceding section describes the relationship between NIP theories, VC finite classes, and falsification, but further argument is required to demonstrate that these phenomena actually appear in the kinds of hypotheses we seek to falsify.

To this end, the following powerful theorem proves the VC finiteness of a very wide class of geometrically-definable hypotheses.

We assume that the reader is familiar with the notion of an analytic function.

\begin{defn} \autocite[]{van1994real}
Let $\mathcal{A}_{[-1,1]}$ be the set of functions $f: [-1,1]\to \RR$ which extends to an analytic function on an open neighborhood $U \supset [-1,1]$. Let $\exp: \RR\to\RR$ be the exponential map $\exp(x) = e^x$.

The restricted analytic exponential real field $\mathcal{R}$ is the structure 
\[\mathcal{R} = (\RR,+,-,0,<, \exp, (f_j)_{f_j\in \mathcal{A}_{[-1,1]}})\]

The theory $\RR_{an,exp}$ is the theory of the structure $\mathcal{R}$.
\end{defn}

In this structure, parametric families of equalities and inequalities between analytic functions on compact rectangular domains are definable. The following result shows that such families have finite VC-dimension:

\begin{thm} \label{thm:ranexp_nip}
$\RR_{an,exp}$ is NIP. Consequently, every first-order definable set in the theory of $\RR_{an,exp}$ has finite VC dimension.
\end{thm}
\begin{proof}
That $\RR_{an,exp}$ is $o$-minimal is a classical theorem of van den Dries and Miller \autocite[]{van1994real}, together with the result that every $o$-minimal theory $T$ is NIP, which can be found in Simon's book \autocite[Theorem A.6]{simon2015guide}.
\end{proof}

This theorem is extremely useful because it implies that not only is any parametric family of algebraic equations over a real field VC finite, but in fact any parametric family of semi-analytic inequalities is VC finite. This is of the utmost importance for examples stemming from physics, as it implies that even definable classes where one includes a model of measurement error can be VC finite. 
As an illustration, we consider the example of the family of \textit{fat lines}. By a fat line I mean a set $\tilde{L}\subseteq \RR^2$ such that $(x,y)\in \tilde{L}(x,y;a,b,c,r)$ just in case $(x,y)$ is most distance $r$ from the line $ax+by+c=0$.

\begin{propo}
The class $\tilde{L}(x,y;a,b,c,r)$ of fat lines in $\RR^2$ has finite VC dimension. \qedhere
\end{propo}
\begin{proof}
By \autoref{thm:ranexp_nip}, any set definable in the theory of $\RR_{an,exp}$ is VC finite. Thus it suffices to give an explicit definition of this family. This is easily done: the formula $\varphi(x,y;a,b,r)$ given by
\[ |ax + by + c|^2 < r(a^2+b^2)  \]
is a formula in the language of $\RR_{an,exp}$ and defines the family of fat lines.
\end{proof}

This suggests an explanation for why many physical theories are so readily falsifiable: many of the predictions of physical theories can be cast in terms of determining membership in a real semianalytic set which expresses being some bounded error away from an analytic or algebraic set.

\subsection{Falsification and Model Theoretic Dividing Lines}

It turns out that other classification-theoretic conditions on $\K$ fail to guarantee falsifiability. Recall from classification theory the following definitions:

\begin{defn}
Let $\varphi(\overline{x};\overline{y})$ be a partitioned formula. We say that
\begin{enumerate}
   \item \autocite[Theorem 2.2.3(2)]{shelah1990classification} $\varphi$ is \textit{stable} if there is no set $(c_n)_{n\in\omega}$ such that for every $k<\omega$,
\[ \{\varphi(x,c_0), \varphi(x,c_1), \cdots,\varphi(x,c_{k-1}), \neg \varphi(x,c_{k}), \neg\varphi(x,c_{k+1}), \cdots\}_{n\in\omega} \]
is consistent.
    \item %\autocite[Definition 2.2]{DZAMONJA2004119} 
    $\varphi$ is NSOP$_1$ if there is no set of tuples
     $\{c_{\sigma}\, | \, \sigma\in 2^{<\omega}\}$ such that
    \begin{enumerate}
        \item (Branch consistency) for every $\tau\in 2^{\omega}$ the set
       \[\{\varphi(x;c_{\tau|_m})\,|\,m\in\omega\} \]
      is consistent, and
     \item (Lateral inconsistency) The set
    \[\{\varphi(x,c_{\sigma\cap\left\langle 1\right\rangle},\varphi(x,c_{\gamma}) \}\]
        is inconsistent for all $\gamma \supseteq \sigma\cap\left\langle 0\right\rangle.$
    \end{enumerate}
\end{enumerate}
A theory $T$ is stable (resp. NSOP$_{1}$) provided every formula in $T$ is stable (resp. NSOP$_1$).
\end{defn}

Je{\v{r}}{\'a}bek \autocite{jevrabek2020recursive}, and, independently, Kruckman and Ramsey \autocite{kruckman2018generic} showed that 

\begin{thm}\label{thm:random_nsop}
Let $\LL$ be a language. Then the model companion of the empty theory $T_{\LL}^{\varnothing}$ exists, is complete, and is
\begin{enumerate}
    \item stable if $\LL$ is unary,
    \item simple if $\LL$ is relational, and
    \item unstable NSOP$_1$ for any non-unary $\LL$.
\end{enumerate}
\end{thm}

An immediate corollary of this theorem is 

\begin{corollary}
Let $\mathcal{L}$ be a language. Then $\forall_1(T_{\LL}^{\varnothing})$ contains only validities, so $\operatorname{Mod}(T_{\LL}^{\varnothing})$ is not a falsifiable class.
\end{corollary}

\begin{proof}
Suppose that $\varphi$ is a $\forall_1$ sentence in $\LL$ that is not a validity. We wish to show that $T\not\models \varphi$. Without loss of generality we assume that \[\varphi = (\forall x_1,\dots x_n) \psi(x_1,\dots,x_n)\]
With $\psi(x_1,\dots,x_n)$ quantifier-free. Since $\varphi$ is not a first-order validity, $\neg \varphi$ is satisfiable and is a $\exists_1$ formula. Let $\MM$ be an $\LL$ structure such that $\MM\models \neg \varphi$. Since $T^{\varnothing}_{\LL}$ is the theory of existentially closed $\LL$-structures, $\MM$ embeds into a model $\MM' \models T$. By construction, $\MM' \models \neg \varphi$, so $\varphi\notin T^{\varnothing}_{\LL}$.
\end{proof}

To be sure, there exist falsifiable NSOP$_1$ classes. The theory of the random graph, for instance, contains the universal theory of graphs, which in particular includes the non-logical validity
\[\forall x\forall y(R(x,y) \leftrightarrow R(y,x)). \]

In short, NIP classes of structures yield examples of falsifiable structures incomparable to the notions we have thus far discussed. Thus, our picture of the relationship between falsifiable classes of falsifiability now looks like (again assuming mild assumptions on $\K$ and $\LL$):

% https://q.uiver.app/?q=WzAsOCxbMSwyLCJcXGJveGVke1xcbWF0aGJie0t9IFxcdGV4dHsgRmFsc2lmaWFibGV9fSJdLFsxLDMsIlxcbWF0aGJie0t9IFxcdGV4dHsgaGFzIG5vbnRyaXZpYWwgVU5DQUYgY29uc2VxdWVuY2V9Il0sWzAsMiwiXFxtYXRoYmJ7S31cXDtcXDsgXFxmb3JhbGxfMS1cXHRleHR7QXhpb21hdGl6YWJsZX0iXSxbMCwxLCJcXG1hdGhiYntLfSBcXHRleHR7IEZJVH0iXSxbMSwxLCJcXG1hdGhiYntLfSBcXHRleHR7IE5TT1B9XzEiXSxbMiwyLCJcXG1hdGhiYntLfSBcXHRleHR7IE5JUH0iXSxbMywyXSxbMSwwLCJcXG1hdGhiYntLfSBcXHRleHR7IFNpbXBsZX0iXSxbMiwwXSxbMSwwXSxbMywyXSxbNCwwLCIiLDIseyJzdHlsZSI6eyJib2R5Ijp7Im5hbWUiOiJzcXVpZ2dseSJ9LCJoZWFkIjp7Im5hbWUiOiJub25lIn19fV0sWzAsNSwiIiwyLHsic3R5bGUiOnsidGFpbCI6eyJuYW1lIjoiYXJyb3doZWFkIn0sImhlYWQiOnsibmFtZSI6Im5vbmUifX19XSxbNyw0XV0=
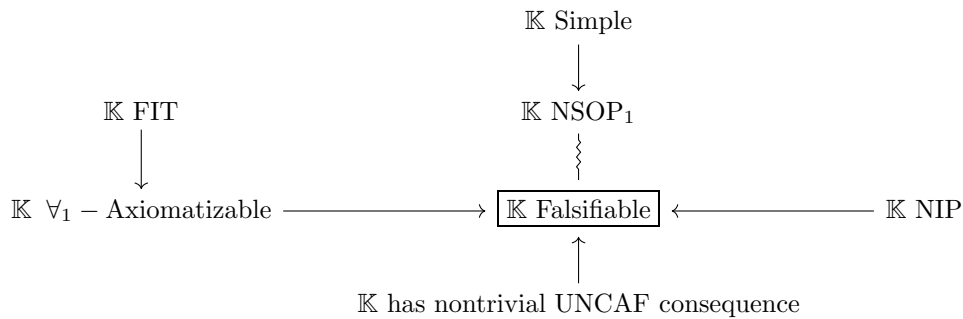
\[\begin{tikzcd}
	& {\mathbb{K} \text{ Simple}} \\
	{\mathbb{K} \text{ FIT}} & {\mathbb{K} \text{ NSOP}_1} \\
	{\mathbb{K}\;\; \forall_1-\text{Axiomatizable}} & {\boxed{\mathbb{K} \text{ Falsifiable}}} & {\mathbb{K} \text{ NIP}} & {} \\
	& {\mathbb{K} \text{ has nontrivial UNCAF consequence}}
	\arrow[from=3-1, to=3-2]
	\arrow[from=4-2, to=3-2]
	\arrow[from=2-1, to=3-1]
	\arrow[squiggly, no head, from=2-2, to=3-2]
	\arrow[tail reversed, no head, from=3-2, to=3-3]
	\arrow[from=1-2, to=2-2]
\end{tikzcd}\]

\section{Conclusion}
Over the course of this paper we investigated various strengthenings of the basic notion of falsifiability that have been defined and studied. The static models of falsifiability typically concern themselves with questions of how close a theory is to being universally axiomatizable; after all, the more universal sentences a theory implies, in principle the more falsfiable the theory becomes. 

We first studied Simon and Groen's \autocite{simon1979fit} notion of FITness which they claim isolates the ideal scientific theories. They show that for pseudoelementary $\K$, FITness implies universal axiomatizability. Generalizing their definition to arbitrary classes of $\LL$-structures $\K$ closed under isomorphism, I show that for finite languages $\K$ being FIT entails that $\K$ is elementary and, in fact, universally axiomatizable. This result substantially generalizes their result over finite languages. 

I then turned my attention to an argument given by Chambers et al. \autocite{chambers2014axiomatic} that argues that being universally axiomatizable is not sufficient grounds to call a theory falsifiable. Instead, they identify the falsifiable sentences with a class of universal sentences they call UNCAF (a \textul{u}niversal \textul{n}egation of a \textul{c}onjunction of \textul{a}tomic \textul{f}ormulas). In turn, I argued that their argument implicitly assumes that the underlying predicates $P\in \LL$ exhibit $\Sigma_1$ behavior, and thus that their argument reaches too far in its conclusions.

As a final foray into the static case of falsification, I considered how falsification intersects with the dividing lines of classification theory. Under very mild restrictions on the language $\LL$, NIP theories entail a number of nontrivial $\forall_1$ sentences, yielding the falsifiability of NIP theories. Of note, in NIP theories each formula $\varphi$ is equipped with a notion of dimension known as the VC-dimension of a class, which in a sense measures the effective falsifiability of membership in the class of hypotheses it defines. 

Finally, we considered recent work of Kruckman and Ramsey \autocite{kruckman2018generic} and, independently, Je{\v{r}}{\'a}bek \autocite{jevrabek2020recursive}, which yield examples of NSOP$_1$ and simple theories which are \textit{un}falsifiable. While there are many NSOP$_1$ theories which are falsifiable, in a sense NIP is individuated among the dividing lines in model theory as a class of highly-falsifiable theories.

\printbibliography
\end{document}